\documentclass{amsart}
\usepackage{graphicx}
\pdfoutput=1
\pdfoutput=1
\pdfoutput=1
\pdfoutput=1
\pdfoutput=1
\newtheorem{thm}{Theorem}[section]

\newtheorem{lem}[thm]{Lemma}

\newtheorem{ex}[thm]{Example}

\theoremstyle{definition}
\newtheorem{defn}[thm]{Definition}
\theoremstyle{remark}
\newtheorem{rem}[thm]{Remark}
\numberwithin{equation}{section}

\begin{document}

\title[ On a consistent estimator of a useful signal]{On a consistent estimator of a useful signal in  Ornstein-Uhlenbeck stochastic model  in  $\mathbb{C}[-l,l[$ }%

\author{Levan Labadze }%
\address{Department of  Mathematics, Georgian Technical University,  Tbilisi , Georgia}%
\email{levanlabadze@yahoo.com}%

\author{Zurab Kvatadze}%
\address{Department of  Mathematics, Georgian Technical University,  Tbilisi , Georgia}%
\email{zurakvatadze@yahoo.com}%

\author{Gogi Pantsulaia}%
\address{Department of  Mathematics, Georgian Technical University,  Tbilisi , Georgia}%
\email{gogipantsulaia@yahoo.com}%

\medskip

\medskip

\thanks{}%


\begin{abstract} ~It is considered a transmittion process  of a useful  signal  in Ornstein-Uhlenbeck stochastic model   in  $\mathbb{C}[-l,l[$   defined by  the stochastic
differential equation
$$
d\Psi(t,x,\omega)=\sum_{n=0}^{2m}
A_n\frac{\partial^{n}}{\partial x^{n}}\Psi(t,x,\omega)dt +\sigma d W(t,\omega)
$$
with initial condition
$$\Psi(0,x,\omega)=\Psi_0(x)  \in FD^{(0)}[-l,l[,
$$
where  $m \ge 1$,   $(A_n)_{0 \le n \le 2m} \in \mathbb{R}^+\times  \mathbb{R}^{2m-1}$,$~((t,x,\omega) \in
[0,+\infty[\times [-l,l[ \times \Omega)$,   $\sigma \in  \mathbb{R}^+$, $\mathbb{C}[-l,l[$  is  Banach space  of  all real-valued bounded continuous functions on $[-l,l[$,   $FD^{(0)}[-l,l[  \subset \mathbb{C}[-l,l[ $ is class of all real-valued bounded continuous functions on  $[-l,l[$  whose Fourier series converges to himself everywhere on $[-l,l[$, $(W(t,\omega))_{t \ge 0}$ is a Wiener process and $\Psi_0(x)$ is a useful signal.

By use  a sequence of  transformed signals $(Z_k)_{k \in N}=(\Psi(t_0,x,\omega_k))_{k \in N}$ at  moment $t_0>0$,  consistent and  infinite-sample  consistent estimates  of the useful signal $\Psi_0$ is constructed  under assumption  that   parameters  $(A_n)_{0 \le n \le 2m}$ and $\sigma$ are known.  Animation and simulation of the  Ornstein-Uhlenbeck process in Banach space $\mathbb{C}[-l,l[$ and   results of calculations of   estimates  of a useful signal  in the same stochastic model  are also  presented.

.
\end{abstract}
\maketitle

\section{Introduction}

Suppose that $\Theta$ is a vector subspace of the
Banach space $\mathbb{C}[-l,l[$ equipped with usual norm,
where  $\mathbb{C}[-l,l[$  denotes  the class of all  bounded continuous functions on $[-l,l[$.

In the information transmitting theory we consider  Ornstein-Uhlenbeck stochastic  system
$$
\xi(t,x,\omega)_{x \in [-l,l[}=e^{t \sum_{n=0}^{2m}
A_n\frac{\partial^{n}}{\partial x^{n}}}\theta(x)_{x \in [-l,l[}
+$$
$$\sigma \int_0^t e^{ (t-\tau) \sum_{n=0}^{2m}
A_n\frac{\partial^{n}}{\partial x^{n}}} \times I_{[-l,l[}(x) d W(\tau,\omega),
\eqno(1.1)
$$
where   $(A_n)_{0 \le n \le 2m} \in \mathbb{R}^+\times  \mathbb{R}^{2m-1}(m \ge 1)$,
$\theta \in \Theta$  is a useful signal, $(W(t, \cdot ))_{t \ge 0}$  is  Winner  processes(the
so-called ``white noises" ) defined on the
probability space $(\Omega,{{\bf F}},P)$,  $(\xi(t,\cdot, \omega))$(equivalently, $\xi(t,x,\omega)_{x \in [-l,l[}$)  is a transformed signal   for  $(t,\omega) \in
[0,+\infty[\times \Omega)$, $ I_{[-l,l[}$ denotes the indicator function of the interval $[-l,l[$.

 Let $\mu$ be a
Borel probability measure on $\mathbb{C}^{[0,+\infty[}[-l,l[$ defined by  generalized
``white noise"
$$
\big(\sigma \int_0^t e^{ (t-\tau) \sum_{n=0}^{2m}
A_n\frac{\partial^{n}}{\partial x^{n}}} \times I_{[-l,l[} d W(\tau,\omega)\big )_{t \ge 0}.
\eqno(1.2)
$$
 Then  we have
$$
 (\forall X)(X \in {\bf B}\big ( \mathbb{C}^{[0,+\infty[}[-l,l[ \big ) \rightarrow
\mu(X)=P(\{ \omega : \omega \in \Omega ~
$$
$$
\&~
 \big( \sigma \int_0^t e^{ (t-\tau) \sum_{n=0}^{2m}
A_n\frac{\partial^{n}}{\partial x^{n}}} \times I_{[-l,l[}d W_k(\tau,\omega)\big)_{t  \ge 0} \in X\})), \eqno(1.3)
$$
where $ {\bf B} \big(\mathbb{C}^{[0,+\infty[}[-l,l[\big)$ is the Borel  $\sigma$-algebra
of subsets of the space  $\mathbb{C}^{[0,+\infty[}[-l,l[$.

Let $\lambda$ be a
Borel probability measure on $\mathbb{C}^{[0,+\infty[}[-l,l[$ defined by transformed signal  $\xi(t,x,\omega)$ that is

$$
 (\forall X)(X \in {\bf B}\big ( \mathbb{C}^{[0,+\infty[}[-l,l[ \big ) \rightarrow
\mu(X)=P(\{ \omega : \omega \in \Omega ~
$$
$$
\&~
 \big( \xi(t,x, \omega) \big)_{t  \ge 0} \in X\})), \eqno(1.4)
$$

  In the information transmitting theory, the general decision is that the Borel probability measure
$\lambda$, defined  by the  transformed signal coincide with $e^{t \sum_{n=0}^{2m}
A_n\frac{\partial^{n}}{\partial x^{n}}}\theta$ shift   $\mu_{\theta}$ of the measure $\mu$ for some $\theta_0 \in \Theta$ provided that
$$
(\exists \theta_0)( \theta_0 \in \Theta \rightarrow (\forall X)(X
\in {\bf B} \big(\mathbb{C}^{[0,+\infty[}[-l,l[ \big )\rightarrow
\lambda(X)=\mu_{\theta}(X)), \eqno(1.5)
$$
where $\mu_{\theta}(\cdot)=\mu(\cdot + \big(e^{t \sum_{n=0}^{2m}
A_n\frac{\partial^{n}}{\partial x^{n}}}\theta \big) _{t \ge 0})$.

Here  we consider a particular case of the above model when
a vector space of  useful signals $\Theta$ coincides with  $ FD^{0}[-l,l[$, where  $ FD^{0}[-l,l[  \subset  \mathbb{C}[-l,l[$   denotes  a vector space of all
bounded  continuous  real-valued functions on    $[-l,l[$    whose Fourier series converges to himself everywhere on $[-l,l[$.

\begin{defn} Following  \cite{Ibram80}, a triplet

$$
( \big( \mathbb{C}^{[0,+\infty[}[-l,l[ \big)^\mathbb{N}  ,{\bf B} \big( \big( \mathbb{C}^{[0,+\infty[}[-l,l[\big )^\mathbb{N}\big),\mu^\mathbb{N}_{\theta} \big)_{\theta
\in \Theta} \eqno(1.6)
$$
 is called a statistical structure described the
stochastic system (1.1).
\end{defn}

\begin{defn} Following  \cite{Ibram80}, a Borel measurable function  $T_n : \big(\mathbb{C}^{[0,+\infty[}[-l,l[ \big)^n \to \Theta~(n \in
\mathbb{N})$ is called a consistent estimate  of a parameter
$\theta$ for the family $(\mu_{\theta}^{\mathbb{N}})_{\theta \in
 \Theta }$ if the condition
$$
\mu_{\theta}^{\mathbb{N}} (\{ (x_k)_{k \in \mathbb{N}} :~(x_k)_{k
\in \mathbb{N}} \in \big(\mathbb{C}^{[0,+\infty[}[-l,l[\big)^{\mathbb{N}}~\&~ \lim_{n \to \infty}||T_n(x_1,
\cdots, x_n)-\theta||=0 \})=1\eqno(1.7)
$$
holds for each $\theta \in  \Theta$, where $||\cdot||$ is a usual norm in   $\mathbb{C}[-l,l[$.
\end{defn}

\begin{defn} Following  \cite{Ibram80}, a Borel measurable function  $T : \big(\mathbb{C}^{[0,+\infty[}[-l,l[ \big)^{\mathbb{N}} \to \Theta$  is called an infinite-sample  consistent estimate  of a parameter
$\theta$ for the family $(\mu_{\theta}^{\mathbb{N}})_{\theta \in
 \Theta}$ if the condition
$$
\mu_{\theta}^{\mathbb{N}} (\{ (x_k)_{k \in \mathbb{N}} :~(x_k)_{k
\in \mathbb{N}} \in \big(\mathbb{C}^{[0,+\infty[}[-l,l[\big)^{\mathbb{N}}~\&~ T((x_k)_{k \in \mathbb{N}})=\theta \})=1
\eqno(1.8)
$$
holds for each $\theta \in  \Theta$.
\end{defn}

The main goal of the present paper is construct  consistent   and  infinite-sample  estimators  of the useful signal  for  the stochastic model  (1.1)  which is a particular case of the  Ornstein-Uhlenbeck process in  $\mathbb{C}[-l,l[$.
Concerning estimations of parameters for another versions of the Ornstein-Uhlenbeck processes the reader  can consult with  \cite{GarbaczewskiOlkiewicz2000},  \cite{Pantsulaia2016},\cite{Labadze2016}, \cite{OrnsteinUhlenbeck1930}.

The rest of the present paper is the following.

Section  2 contains some auxiliary notions and fact from  theories of  ordinary   and stochastic differential  equations.

In Section 3 we  present our main results.

In Section  4 we  present animations  and simulations  of the  Ornstein-Uhlenbeck process in  $\mathbb{C}[-l,l[$ and
present results of calculations of the estimator of a useful signal when parameters  $(A_n)_{0 \le n \le 2m}$,$\sigma$ and  a sample of  transformed signals  at moment $t_0>0$   defined  by $(1.1)$  are known.

In Section  5  we consider discussion and conclusion.

\section{Materials and methods}

We  begin  this section by  a short description of a certain result  concerning  a solution of some
differential  equations  with initial  value problem obtained  in  the  paper \cite{Pantsulaia2016-1}. Further, by  use  this approach and  technique  developed  in \cite{Protter2004},  their some applications  for a solution  of the  Ornstein-Uhlenbeck stochastic differential equation in  $\mathbb{C}[-l,l[$   are  obtained.  At end of  this section, well known  Kolmogorov Strong Law of Large Numbers is  presented.

\begin{lem}[\cite{Pantsulaia2016-1}, Corollary 2.1, p. 6]
 ~For $m \ge 1$, let us consider a linear partial
differential equation
$$
\frac{\partial}{\partial t}\Psi(t,x)=\sum_{n=0}^{2m}
A_n\frac{\partial^{n}}{\partial x^{n}}\Psi(t,x)~((t,x) \in
[0,+\infty[\times [-l,l[)\eqno(2.1)
$$
with initial condition
$$\Psi(0,x)=\frac{c_0}{2}+\sum_{k=1}^{\infty}c_k
\cos\Big(\frac{k \pi x}{l}\Big)+d_k \sin \Big(\frac{k \pi
x}{l}\Big) \in FD^{(0)}[-l,l[.\eqno(2.2)
$$
If $(\frac{c_0}{2}, c_1, d_1, c_2, d_2, \dots )$ is such a
sequence of real numbers that a series $\Psi(t,x)$ defined by
$$ \Psi(t,x)=\frac{e^{tA_0}c_0}{2}+\sum_{k=1}^{\infty} e^{\sigma_k t}\Big((c_k\cos (\omega_k t)+
$$
$$d_k\sin(\omega_k t))\cos(\frac{k\pi x}{l})+
(d_k\cos (\omega_k t)-c_k \sin(\omega_k t))\sin(\frac{k\pi
x}{l})\Big)\eqno(2.3)  $$
 belongs to the class  $FD^{(2m)}[-l,l[$ as a series  of a variable  $x$ for all $t \ge 0$, and
 is differentiable term by term  as a series  of a variable  $t$ for all $x \in [-l,l[$, then  $\Psi$ is a
 solution of $(2.1)$-$(2.2)$.
\end{lem}

\medskip

By use an approach developed in \cite{Protter2004}, we get the validity of the following assertion.

\medskip

\begin{lem}
 ~For $m \ge 1$, let us consider Ornstein-Uhlenbeck process in  $\mathbb{C}[-l,l[$ defined by  the stochastic
differential equation
$$
d\Psi(t,x,\omega)=\sum_{n=0}^{2m}
A_n\frac{\partial^{n}}{\partial x^{n}}\Psi(t,x,\omega)dt +\sigma d W(t,\omega) I_{[-l,l[}(x) ~((t,x,\omega) \in
[0,+\infty[\times [-l,l[ \times \Omega)\eqno(2.4)
$$
with initial condition
$$\Psi(0,x,\omega)=\Psi_0(x), \eqno(2.5)
$$
where
 $(A_n)_{0 \le n \le 2m} \in \mathbb{R}^+\times  \mathbb{R}^{2m-1}$,$ (W(t,\omega))_{t \ge 0}$ is a Wiener process and
 $$
\Psi_0(x)=\frac{c_0}{2}+\sum_{k=1}^{\infty}c_k
\cos\Big(\frac{k \pi x}{l}\Big)+d_k \sin \Big(\frac{k \pi
x}{l}\Big) \in FD^{(0)}[-l,l[.  \eqno(2.6)
$$

If $(\frac{c_0}{2}, c_1, d_1, c_2, d_2, \dots )$ is such a
sequence of real numbers that a series
$$ \frac{e^{tA_0}c_0}{2}+\sum_{k=1}^{\infty} e^{\sigma_k t}\Big((c_k\cos (\omega_k t)+
$$
$$d_k\sin(\omega_k t))\cos(\frac{k\pi x}{l})+
(d_k\cos (\omega_k t)-c_k \sin(\omega_k t))\sin(\frac{k\pi
x}{l})\Big)   +\sigma\int_0^t e^{(t-\tau)A_0} d W_{\tau}(\omega)I_{[-l,l[}(x) \eqno(2.7)  $$
 belongs to the class  $FD^{(2m)}[-l,l[$ as a series  of a variable  $x$ for all $t \ge 0, \omega \in \Omega$, and
 is differentiable term by term  as a series  of a variable  $t$ for all $x \in [-l,l[, \omega \in \Omega$, then
the solution of $(2.5)$-$(2.6)$  is given by

$$\Psi(t,x,\omega) =e^{t \sum_{n=0}^{2m}
A_n\frac{\partial^{n}}{\partial x^{n}}}(\Psi(0,x,\omega) ) +\sigma \int_0^te^{(t-s) \sum_{n=0}^{2m}
A_n\frac{\partial^{n}}{\partial x^{n}}} dW(s,\omega).
$$
\end{lem}

\begin{proof}

Putting $\mathbb{A}=\sum_{n=0}^{2m}
A_n \partial^{n}/\partial x^{n}$
and
$
f(t, \Psi(t,x,\omega))=e^{-t\mathbb{A} }\Psi(t,x,\omega),
$
we get
$$
df(t, \Psi(t,x,\omega))=-\mathbb{A} e^{-t \mathbb{A} } \Psi(t,x,\omega)dt+e^{-t \mathbb{A} }d\Psi(t,x,\omega)
$$
$$
=-\mathbb{A} e^{-t \mathbb{A} } \Psi(t,x,\omega)dt+e^{-t \mathbb{A} }(\mathbb{A}\Psi(t,x,\omega)+\sigma dW(t,\omega))=
\sigma e^{-t \mathbb{A} }dW(t,\omega).
$$
By integration of both sides we get
$$
f(t, \Psi(t,x,\omega))-f(0, \Psi(0,x,\omega))=\sigma \int_0^t e^{-\tau\mathbb{A} } d W(\tau,\omega)
 $$
which implies
$$
e^{-t\mathbb{A} }\Psi(t,x,\omega)-e^{-0\mathbb{A} }\Psi(0,x,\omega)=\sigma \int_0^t e^{-\tau\mathbb{A} } d W(\tau,\omega).
$$
Now we get
$$
e^{t\mathbb{A} }(e^{-t\mathbb{A} }(\Psi(t,x,\omega))-e^{t\mathbb{A} }(e^{-0\mathbb{A} }(\Psi(0,x,\omega)))=e^{t\mathbb{A} }(\sigma \int_0^t e^{-\tau\mathbb{A} } d W(\tau,\omega)),
$$
which is equivalent to the  equality
$$
\Psi(t,x,\omega))=e^{t\mathbb{A} }(\Psi(0,x,\omega))+\sigma \int_0^t e^{(t-\tau)\mathbb{A} } I_{[-l,l[}(x) d W(\tau,\omega)).
$$

\end{proof}

\begin{rem} ~Under condition of Lemma 2.2  we have
$$\Psi(t,x,\omega)=\frac{e^{tA_0}c_0}{2}+\sum_{k=1}^{\infty} e^{\sigma_k t}\Big((c_k\cos (\omega_k t)+
$$
$$d_k\sin(\omega_k t))\cos(\frac{k\pi x}{l})+
(d_k\cos (\omega_k t)-c_k \sin(\omega_k t))\sin(\frac{k\pi
x}{l})\Big)  +
$$
$$
\sigma \int_0^t e^{(t-\tau)A_0} I_{[-l,l[}(x) d W(\tau,\omega)).   \eqno(2.8)  $$
 \end{rem}

\begin{lem}  Under  conditions  of  Lemma 2.2, the following  conditions are valid:

(i)~ $E \Psi(t,x, \cdot)=  \frac{e^{tA_0}c_0}{2}+\sum_{k=1}^{\infty} e^{\sigma_k t}\Big((c_k\cos (\omega_k t)+d_k\sin(\omega_k t))\cos(\frac{k\pi x}{l})$

$+(d_k\cos (\omega_k t)-c_k \sin(\omega_k t))\sin(\frac{k\pi x}{l})\Big);$

(ii) ~  $\mbox{cov}(\Psi(s,x,\cdot),\Psi(t,x,\cdot))=\frac{\sigma^2}{2A_0}\left( e^{-A_0(t-s)} - e^{-A_0(t+s)} \right) ;$

(iii) ~$\mbox{var}(\Psi(s,x,\cdot)) = \frac{\sigma^2}{2 A_0}\left( 1 - e^{-2 A_0 s} \right);$

\end{lem}

\begin{proof}  The validity  of the item (i) is obvious.  In order to prove the validity of the items (ii)-(iii),  we can use the Ito  isometry  to calculate the covariance function  by
$$
\mbox{cov}(\Psi(s,x,\cdot),\Psi(t,x,\cdot)) = E[(\Psi(s,x,\cdot) - E[\Psi(s,x,\cdot)])(\Psi(s,t,\cdot) - E[\Psi(s,t,\cdot)])]
$$
$$
 = E \left[ \int_0^s \sigma  e^{A_0 (u-s)}\, dW(u,\omega) \int_0^t \sigma  e^{A_0 (v-t)}\, dW(v,\omega)\right]
$$
$$
 = \sigma^2 e^{-A_0(s+t)}E \left[ \int_0^s  e^{A_0 u}\, dW(u,\omega) \int_0^t  e^{A_0 v}\, dW(v,\omega) \right]
$$
$$
= \frac{\sigma^2}{2A_0} \, e^{-A_0 (s+t)}(e^{2A_0 \min(s,t)}-1).
$$
Thus if $s<t$(so that $min(s,t)=s$), then we have
$$\mbox{cov}(\Psi(s,x,\cdot),\Psi(t,x,\cdot)) = \frac{\sigma^2}{2A_0}\left( e^{-A_0(t-s)} - e^{-A_0(t+s)} \right).
$$
Similarly, if $s=t$ (so that $min(s,t)=s)$, then we have
$$\mbox{var}(\Psi(s,x,\cdot)) = \frac{\sigma^2}{2A_0}\left( 1 - e^{-2A_0 s} \right).
$$

\end{proof}

In the next section  we will need the well known fact from the
probability theory (see, for example, \cite{Sh80}, p. 390).
\medskip

\begin{lem}  (Kolmogorov's strong law of large numbers)
 Let $X_1, X_2, ...$  be a  sequence of independent identically
distributed random variables  defined on the probability space
$(\Omega, \mathcal{F},P)$. If  these random variables have a
finite expectation $m$ (i.e., $E(X_1) = E(X_2) = ... = m <
\infty$), then the following condition
$$
P(\{ \omega : \lim_{n \to \infty} n^{-1}\sum_{k=1}^nX_k(\omega)=m
\})=1  \eqno(2.9)
$$
holds  true.
\end{lem}

\section{Results}

In this  section, by the use  of Kolmogorov Strong Law of Large Numbers   we construct   a consistent   and an infinite-sample  consistent   estimators   of a useful signal which is
transmitted    by the Ornstein-Uhlenbeck stochastic  system (1.1).

\begin{thm} Let consider  $\mathbb{C}[-l,l[$-valued  stochastic process  $(\xi(t,x, \omega)_{x \in [-l,l[}  )_{t \ge 0}$ defined  by
$$
\xi(t,x, \omega)_{x \in [-l,l[}=e^{t\mathbb{A} }(\theta(x)_{x \in [-l,l[} )+\sigma \int_0^t e^{(t-\tau)\mathbb{A} } I_{[-l,l[}(x)_{x \in [-l,l[} d W(\tau,\omega),\eqno(3.1)
$$
where $\theta \in FD^{(0)}$ and  $\mathbb{A}=\sum_{n=0}^{2m}
A_n\frac{\partial^{n}}{\partial x^{n}}$. Assume that  all  conditions of Lemma 2.9 are satisfied.
For a fixed  $t=t_0>0$, we denote  by $\mu_{\theta}$ a probability measure in $\mathbb{C}[-l,l[$   defined by the  random  element $(\Xi(t_0,\omega))$.
For  each $(Z_k)_{k \in \mathbb{N}} \in  \big(\mathbb{C}[-l,l[\big)^{\mathbb{N}}$ we put
$$
T_n((Z_k)_{k \in \mathbb{N}})=e^{-t_0\mathbb{A} }(\frac{\sum_{k=1}^n Z_k}{n}). \eqno(3.2)
$$
Then $T_n$ is a consistent estimate of a  useful signal $\theta$ provided
that
$$
\mu_{\theta}^{\mathbb{N}}\{ (Z_k)_{k \in \mathbb{N}}: \lim_{n \to \infty}|| T_n((Z_k)_{k \in \mathbb{N}})-\theta||=0\} =1 \eqno(3.3)
$$
for each $\theta \in FD^{(0)}[-l,l[.$
\end{thm}
\begin{proof} For each $X \in B\big(\mathbb{C}[-l,l[\big)$ we have
$$
\mu_{\theta}(X)=P\{\omega: \Xi(t_0,\omega) \in X\}=P\{\omega: e^{t_0\mathbb{A} }(\theta)+\sigma \int_0^{t_0} e^{(t_0-\tau)\mathbb{A} } I_{[-l,l[} d W(\tau,\omega)) \in X\}=
$$
$$
P\{\omega:\sigma \int_0^{t_0} e^{(t_0-\tau)A_0}  d W(\tau,\omega))  I_{[-l,l[} \in \big(X- e^{t_0\mathbb{A} }(\theta) \big)   \cap \{ \alpha  I_{[-l,l[} : \alpha \in \mathbb{R} \} \}.
$$
We have
$$
\mu_{\theta}^{\mathbb{N}}\{ (Z_k)_{k \in \mathbb{N}}: \lim_{n \to \infty}|| T_n((Z_k)_{k \in \mathbb{N}})-\theta||=0\} =
$$

$$
\mu_{\theta}^{\mathbb{N}}\{ (Z_k)_{k \in \mathbb{N}}: \lim_{n \to \infty}|| e^{-t_0\mathbb{A} }(\frac{\sum_{k=1}^n Z_k}{n})-\theta||=0\} =
$$

$$
\mu_{\theta}^{\mathbb{N}}\{ (Z_k)_{k \in \mathbb{N}}: \lim_{n \to \infty}|| \frac{\sum_{k=1}^n Z_k}{n}- e^{t_0\mathbb{A} }(\theta)||=0\} =
$$

$$
\mu_{\theta}^{\mathbb{N}}\{ (Z_k)_{k \in \mathbb{N}}:  Z_k \in \{ \alpha_k I_{[-l,l[} : \alpha_k \in \mathbb{R}\}+ e^{t_0\mathbb{A} }(\theta) ~\&~ \lim_{n \to \infty}|| \frac{\sum_{k=1}^n Z_k}{n}- e^{t\mathbb{A} }(\theta)||=0\} =
$$

$$
\mu_{\theta}^{\mathbb{N}}\{ (Z_k)_{k \in \mathbb{N}}: ( \exists (\beta_k)_{k \in N}  \in \mathbb{R}^{\infty})( Z_k =\beta_k I_{[-l,l[} +
$$
$$e^{t_0\mathbb{A} }(\theta) ~\&~ \lim_{n \to \infty}|| \frac{\sum_{k=1}^n Z_k}{n}- e^{t_0 \mathbb{A}}(\theta)||$$
$$=0\} =\mu_{\theta}^{\mathbb{N}}\{ (Z_k)_{k \in \mathbb{N}}: ( \exists (\beta_k)_{k \in N}  \in \mathbb{R}^{\infty})( Z_k =\beta_k I_{[-l,l[} +
$$
$$e^{t_0\mathbb{A} }(\theta) ~\&~ \lim_{n \to \infty}|| \frac{\sum_{k=1}^n \beta_k}{n} I_{[-l,l[}||
$$
$$=0\} =\mu_{\theta}^{\mathbb{N}}\{ (Z_k)_{k \in \mathbb{N}}: ( \exists (\beta_k)_{k \in N}  \in \mathbb{R}^{\infty})( Z_k =\beta_k I_{[-l,l[} +
$$
$$e^{t_0\mathbb{A} }(\theta) ~\&~ \lim_{n \to \infty}|\frac{\sum_{k=1}^n \beta_k}{n} |=0\} =
$$
$$
{\gamma_{(0,s)}}^{\mathbb{N}}\{ \exists (\beta_k)_{k \in N}  \in \mathbb{R}^{\infty} :  \lim_{n \to \infty}|\frac{\sum_{k=1}^n \beta_k}{n} |=0\} =1,
$$
where $\gamma_{(0,\sigma)}$ denotes the Gaussian measure in  $\mathbb{R}$  with the mean  $0$ and  the variance  $s^2=\frac{\sigma^2}{2A_0} ( 1 - e^{-2A_0 t_0})$.  The validity of the last equality is a direct consequence of  Lemmas 2.9-2.10.

\end{proof}

\begin{thm}(Continue) Let  $\theta^{*} \in  FD^{(0)}$.  For   $(Z_k)_{k \in \mathbb{N}} \in  \big(\mathbb{C}[-l,l[\big)^{\mathbb{N}}$  we put $T(  (Z_k)_{k \in \mathbb{N}}) = lim_{n \to \infty}T_n(  (Z_k)_{k \in \mathbb{N}}),$ if  the sequence   $(T_n(  (Z_k)_{k \in \mathbb{N}}))_{n \to \mathbb{N}}$ is convergent  and this limit belongs to the class $ FD^{(0)}$, and $T(  (Z_k)_{k \in \mathbb{N}})=\theta^{*} $,  otherwise. Then $T : \big(\mathbb{C}[-l,l[\big)^{\mathbb{N}} \to  FD^{(0)}$ is an infinite-sample consistent estimate of  a useful signal  $\theta \in   FD^{(0)}$    with respect to family $\mu_{\theta}^{\mathbb{N}}$  provided that
the condition
$$
\mu_{\theta}^{\mathbb{N}} (\{ (Z_k)_{k \in \mathbb{N}} :~(Z_k)_{k
\in \mathbb{N}} \in \big(\mathbb{C}^{[0,+\infty[}[-l,l[\big)^{\mathbb{N}}~\&~ T((Z_k)_{k \in \mathbb{N}})=\theta \})=1 \eqno(3.4)
$$
holds for each $\theta \in  FD^{0}[-l,l[$.
\end{thm}

\begin{proof}For  $\theta \in  FD^{0}[-l,l[$, by the use  the result of Theorem 3.1   we get

$$
\mu_{\theta}^{\mathbb{N}} (\{ (Z_k)_{k \in \mathbb{N}} :~(Z_k)_{k
\in \mathbb{N}} \in \big(\mathbb{C}^{[0,+\infty[}[-l,l[\big)^{\mathbb{N}}~\&~ T((Z_k)_{k \in \mathbb{N}})=\theta \}) \ge
$$
$$
\mu_{\theta}^{\mathbb{N}} (\{ (Z_k)_{k \in \mathbb{N}} :~(Z_k)_{k
\in \mathbb{N}} \in \big(\mathbb{C}^{[0,+\infty[}[-l,l[\big)^{\mathbb{N}}~\&~\lim_{n \to \infty} T_n((Z_k)_{k \in \mathbb{N}})=\theta\})=1.
$$
\end{proof}

\section{Animation and simulation of the  Ornstein-Uhlenbeck process in $\mathbb{C}[-l,l[$ and  an estimation of a useful signal}

There  exist  many approaches   and  codes  in Matlab  which can be used  for  simulations of  various  stochastic  processes  which   are  described  by   Ornstein-Uhlenbeck  stochastic differential  equations(see, for example  \cite{Gillespie1996}, \cite{Pantsulaia2016}, \cite{Labadze2016}, \cite{Smith2010}).  Our main attention is  devoted to  animation and  simulation of  a  useful signal  transmitting    processes   which  all  are   described by the  Ornstein-Uhlenbeck  stochastic system  (1.1).  We are going  also to demonstrate whether works the statistic $T_n$  constructed in Theorem 3.1.  In this context we present  some codes in Matlab which are described by the following examples. In all examples we assume that  $(\Omega, \mathbb{F},P)=({\bf R}^N,  \mathbb{B}({\bf R}^ N), \gamma^N)$, where  $\gamma^N$ denotes  $N$-power of the linear  standard Gaussian measure $\gamma$ in $R$.

\begin{ex}  Below we give an animation of the Ornstein-Uhlenbeck process in $\mathbb{C}[0,\pi[$ which is defined by the following stochastic differential equation
$$
d \Psi(t,x,\omega )=2 \Psi(t,x,\omega)dt -10
\frac{\partial}{\partial x}\Psi(t,x,\omega)dt + 1.174  d W(t,\omega),\eqno(4.1) $$

with initial condition
$$\Psi(0,x,\omega)=1/2 +15 cos x  +3cos 3x + cos 8x  +5 sin 3x + 15 sin 5x ,  \eqno(4.2)
$$
$$ ~(t,x,\omega ) \in
[0,\frac{\pi}{7}[\times [-\pi,\pi[  \times R^{N}.
$$
\begin{figure}[h]
\center{\includegraphics[width=0.8 \linewidth]{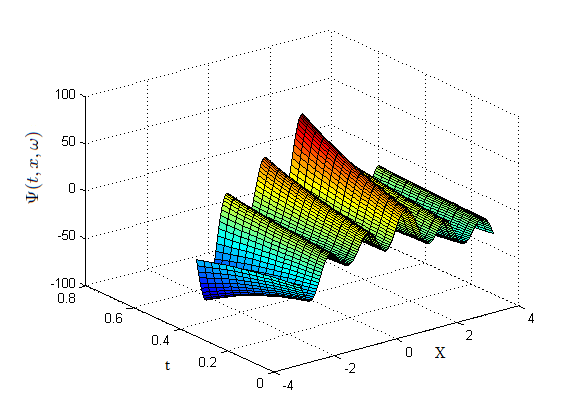}}
\caption{ Picture  from an animation of  Ornstein-Uhlenbeck  stochastic process  defined  by stochastic differential equation (4.1) with initial value problem (4.2)
 }
\label{ris:image}
\end{figure}

$>> N=10000;$

$s=1.174;$

$x1 =\mbox{ random}(\mbox{'Normal'},0,1,N,  1000);$

$A1=[-10,0,0,0,0,  0,0,0,0,0,  0,0,0,0,0,  0,0,0,0,0,  0,0]; $

$C1=[15,0,3,0,0  ,0,0,1,0,0, 0,0,0,0,0, 0,0,0,0,0]; $

$D1=[0,0,5,0,15,0,0,0,0,0, 0,0,0,0,0, 0,0,0,0,0];  $

$A10=2;   A20=0; C10=1; $

$\mbox{for}~ k=1:20$

$S1(k)=A10;   S2(k)=A20;$

$\mbox{for}~ n=1:10$

$S1(k)=S1(k)+(-1)^(n)*A1(2*n)*k^(2*n);$

$\mbox{end}$

$\mbox{end}$

$\mbox{for}~ k=1:20$

$O1(k)=0; $

$\mbox{end}$

$\mbox{for}~ k=1:20$

$\mbox{for} n=1:10$

$O1(k)=O1(k) +(-1)^n*A1(2*n+1)*k^(2*n+1);$

$\mbox{end}$

$\mbox{end}$

$[T1,X1]=\mbox{meshgrid}(0:( pi/100):2*pi, -pi : ( pi/100):pi);$

$\mbox{for}~ m=1:N$

$Z1=0.5* C10*exp(T1.*A10) +  s*x1(m,1)*( (s/sqrt(2*A10)) *exp(T1 *(-1)*A10).* exp(T1*2*A10)-1);$

$\mbox{for}~  k=1:20$

$Z1=Z1+ C1(k)*exp(T1*S1(k)).*cos(X1.*k).* cos(T1*O1(k))+D1(k)*exp(T1*S1(k)).*cos(X1.*k).* sin(T1*O1(k))+ D1(k)*exp(T1*S1(k)).* sin(X1.*k).* cos(T1*O1(k))- $

$ C1(k)*exp(T1*S1(k)).* sin(X1.*k).* sin(T1*O1(k))+  s*sqrt(2)*x1(m,k+1)* sin(pi*k*(exp(2*A10*T1)-1))/(pi*k);$

$\mbox{end}$

$\mbox{surf}~ (X1,T1,Z1)$

$\mbox{drawnow;}$

$\mbox{pause}(1);$

$\mbox{end}$

\end{ex}

\begin{ex}   Let consider  the Ornstein-Uhlenbeck  stochastic differential equation
$$
d\Psi(t,x,\omega )=2 \Psi(t,x,\omega)dt-
\frac{\partial}{\partial x}\Psi(t,x,\omega) dt
+  \sigma d W(t,\omega),\eqno(4.3)
$$

with initial condition
$$\Psi(0,x,\omega)=1/2+5cos x  +5 cos 5x ,   \eqno(4.4)
$$
$$ ~(t,x,\omega ) \in
[0,\frac{\pi}{7}[\times [-\pi ,\pi[\times R^{\infty}.
$$
Below we present the programm in Matlab which draw  a sample of the size  4  which are results of observations  to solutions of  the Ornstein-Uhlenbeck  stochastic differential equation $(4.3)-(4.4)$
at moment $t=\pi/7$.

\begin{figure}[h]
\center{\includegraphics[width=1.0\linewidth]{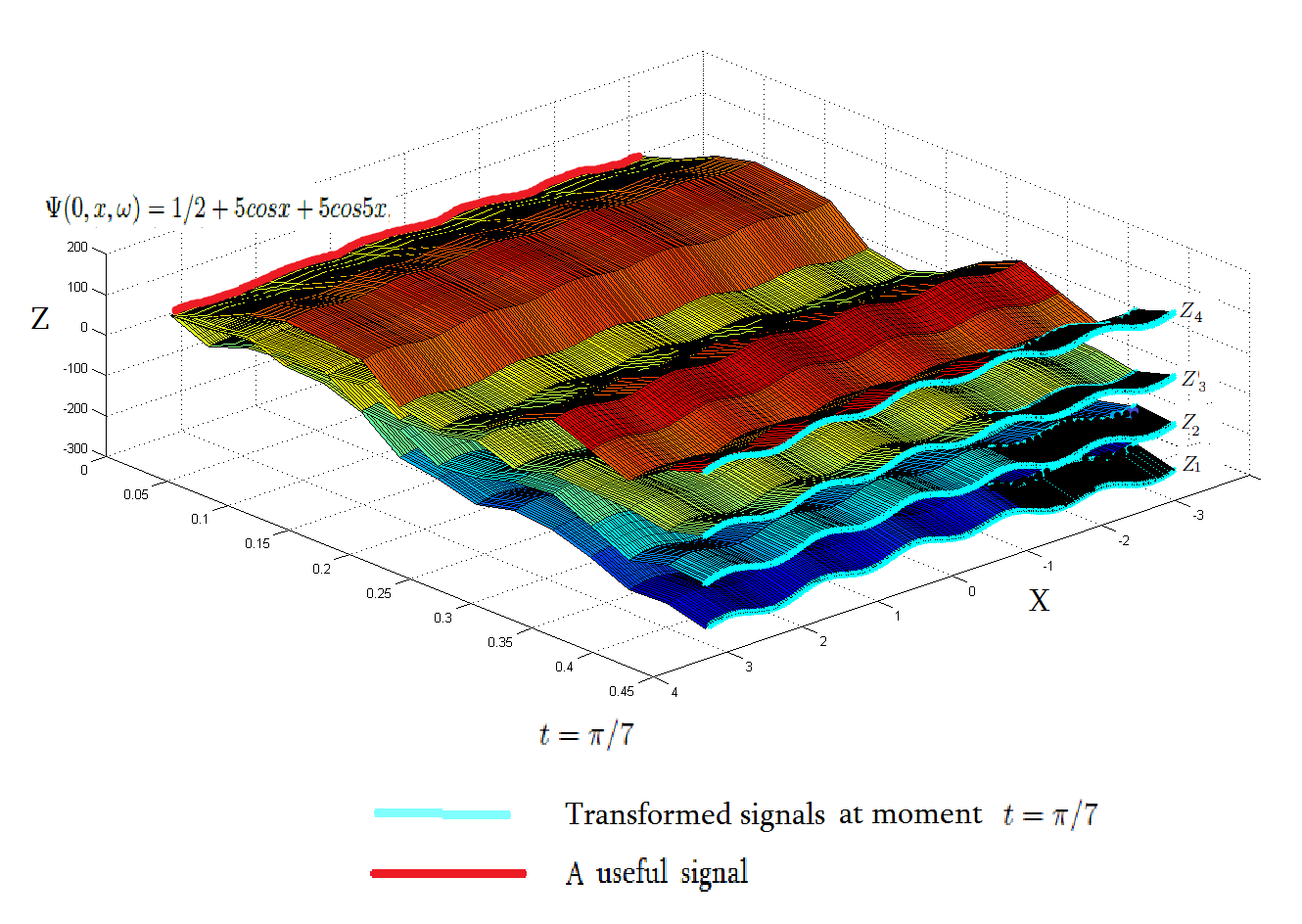}}
\caption{ A sample of the size  4  which   are results of observations  to the solutions  of  the Ornstein-Uhlenbeck  stochastic differential equation (4.3)-(4.4)  at moment $t=\pi/7$ when $\sigma=150$
 }
\label{ris:image}
\end{figure}

$>> N=4; $

$x1 = \mbox{random}('\mbox{Normal}',0,1,N,  1000);$

$A1 = [-1; 0; 0; 0; 0; 0; 0; 0; 0; 0; 0; 0; 0; 0; 0; 0; 0; 0; 0; 0; 0; 0];$

$C1 = [5; 0; 0; 0; 0; 0; 0; 0; 0; 0; 0; 0; 0; 0; 0; 0; 0; 0; 0; 0];$

$D1 = [0; 0; 0; 0; 5; 0; 0; 0; 0; 0; 0; 0; 0; 0; 0; 0; 0; 0; 0; 0];$

$A10 = 2; C10 = 1; s = 150;$

$for k=1:20$

$S1(k)=A10;$

$\mbox{for}~ n=1:10$

$S1(k)=S1(k)+(-1)^(n)*A1(2*n)*k^(2*n);$

$\mbox{end}$

$\mbox{end}$

$for k=1:20$

$O1(k)=0;$

$\mbox{end}$

$\mbox{for}~ k=1:20$

$\mbox{for}~ n=1:10$

$O1(k)=O1(k) +(-1)^n*A1(2*n+1)*k^(2*n+1);$

$\mbox{end}$

$\mbox{end}$

$[T1,X1]=meshgrid(0:( pi/100):pi/7, -pi : ( pi/100):pi);$

$for m=1:N$

$Z_{m}=0.5* C10*exp(T1*A10) + (s*x1(1,1)/sqrt(2*A10))*exp(-A10*T1).* (exp(2*A10*T1)-1);$

$\mbox{end}$

$for m=1:N$

$for  k=1:20$

$Z_{m}=Z_{m}+ C1(k)*exp(T1*S1(k)).*cos(X1.*k).* cos(T1*O1(k))+D1(k)*exp(T1*S1(k)).*cos(X1.*k).* sin(T1*O1(k))+ D1(k)*exp(T1*S1(k)).* sin(X1.*k).* cos(T1*O1(k))- $

$ C1(k)*exp(T1*S1(k)).* sin(X1.*k).* sin(T1*O1(k))+  (s/(sqrt(2*A10 )*pi*k))* sqrt(2)* x1(m,k+1)*exp(-A10*T1).* sin(pi*k*(exp(2*A10*T1)-1));$

$\mbox{end}$

$\mbox{end}$

$\mbox{surf}(X1,T1,Z_{1})$

$\mbox{hold ~on}$

$\mbox{surf}(X2,T2,Z_{2})$

$\mbox{hold~ on}$

$\mbox{surf}(X3,T3,Z_{3})$

$\mbox{hold ~on}$

$\mbox{surf}(X4,T4,Z_{4})$

$\mbox{hold ~off}$

$>>$

\end{ex}

\begin{ex}  Suppose that we have  a sample  $(Z_i)_{1 \le i \le n} \in ( FD^{(0)}[-\pi,\pi[)^n$ of size $n$. In our simulation, we have  that  $$Z_i=(\Psi(t_0,x,  \omega^{(i)}))_{x \in [-\pi,\pi[}$$
for $1 \le i \le n$, where $ \omega^{(i)}=(x^{(i)}_j)_{j \in N} \in {\bf R}^N $  is $\gamma$ -uniformly distributed  sequence  in ${\bf R}$ for each $1 \le i \le n$. For example, we can put
$x^{(i)}_j=\Phi^{-1}(\{j \sqrt{p_i} \})$,   where $p_i$ is $i$-th simple natural number for $i \in N$, $j \in N$, $\{\cdot\}$ denotes a fractal part of the real number and
$\Phi(t)=\frac{1}{\sqrt{2\pi}}\int_{-\infty}^t e^{-\frac{y^2}{2}}dy$  for $t \in R$. In that case we can simulate Wiener trajectory  $W(t,  (x^{(i)}_j)_{j \in N})$  as follows:
$$
W(t, (x^{(i)}_j)_{j \in N})   =x^{(i)}_0 t+ \sqrt{2}\sum_{n=1}^{\infty} x^{(i)}_n\frac{\sin \pi n t}{\pi n}.  \eqno (4.5)$$

Since

$$
\sigma \int_0^{t_0} e^{(t_0-\tau)\mathbb{A} } d W(\tau, (x^{(i)}_j)_{j \in N}))=  \frac{\sigma}{2A_0}e^{-A_0t_0}( x^{(i)}_0(e^{2A_0t_0}-1)+
$$
$$\sqrt{2}\sum_{n=1}^{\infty} x^{(i)}_n \frac{\sin (\pi n(e^{2A_0t_0}-1)) }{\pi n}),\eqno (4.6)
$$
we can simulate $Z_i$ as follows
$$Z_i= (\Psi(t_0,x,  (x^{(i)}_j)_{j \in N}))_{x \in [-\pi,\pi[}  = \frac{e^{tA_0}c_0}{2}+\sum_{k=1}^{\infty} e^{\sigma_k t}\Big((c_k\cos (\omega_k t)+
$$
$$d_k\sin(\omega_k t))\cos(\frac{k\pi x}{l})_{x \in[-\pi,\pi[}+
(d_k\cos (\omega_k t)-c_k \sin(\omega_k t))\sin(\frac{k\pi
x}{l})_{x \in [-\pi,\pi[}\Big)+$$
$$
 \frac{\sigma}{2A_0}e^{-A_0t_0}( x^{(i)}_0(e^{2A_0t_0}-1)+\sqrt{2}\sum_{n=1}^{\infty} x^{(i)}_n \frac{\sin (\pi n(e^{2A_0t_0}-1)) }{\pi n}).\eqno (4.7)
$$

  Suppose we want to estimate a useful signal $\Psi_0(x)_{x \in [-l,l[} \in FD^{(0)}$
defined by  $$
\Psi_0(x)_{x \in [-\pi,\pi[} =\frac{c_0}{2}+\sum_{k=1}^{\infty}c_k
\cos\Big(\frac{k \pi x}{l} \Big)_{x \in [-\pi,\pi[}+d_k \sin \Big(\frac{k \pi
x}{l} \Big)_{x \in [-\pi,\pi[}.\eqno (4.8)
$$

For a function $f \in FD^{(0)}([-\pi,\pi[ )$  we put :

$$\tilde{c}_0(f)=\frac{1}{\pi} \int_{-\pi}^{\pi}f(x)dx;\eqno (4.9)$$

$$ \tilde{c}_k(f)= \frac{1}{\pi} \int_{-\pi}^{\pi}\cos(kx)f(x)dx (k \in \mathbb{N});\eqno (4.10)$$

$$\tilde{d}_k(f)= \frac{1}{\pi} \int_{-\pi}^{\pi}\sin(kx)f(x)dx(k \in \mathbb{N}).\eqno (4.11)$$

Suppose that  all conditions  of Theorem  3.1 are satisfied. Then  following Theorem 3.1, an estimator $T_n$ of the  useful signal $\Psi_0$ is given by

$$ T_n((Z_i)_{1 \le i \le n})=\frac{e^{-t_0A_0}\tilde{c}_0(\frac{\sum_{i=1}^nZ_i}{n})}{2}+\sum_{k=1}^{\infty} e^{-\sigma_k t_0}\Big((\tilde{c}_k(\frac{\sum_{i=1}^n Z_i}{n})\cos (\omega_k t_0)+
$$
$$\tilde{d}_k(\frac{\sum_{i=1}^n Z_i}{n})\sin(\omega_k t_0))\cos(k x)_{x \in[-\pi,\pi[}+
(\tilde{d}_k(\frac{\sum_{i=1}^n Z_i}{n})\cos (\omega_k t_0)-
$$
$$
\tilde{c}_k(\frac{\sum_{i=1}^n Z_i}{n}) \sin(\omega_k t_0))\sin(kx)_{x \in[-\pi,\pi[}\Big)  \eqno(4.12)  $$

\medskip

Our next programm  draws  results of calculations of the estimate $T_{10}$ of  a useful signal when we have a sample of size 10  which are results of observations to the
  solutions   of  the Ornstein-Uhlenbeck  stochastic differential equation $(4.3)-(4.4)$ at moment $t_0=\pi/7$  and  $\sigma \in \{150; 1500; 7500; 15000 \} .$

\begin{figure}[h]
\center{\includegraphics[width=1.05\linewidth]{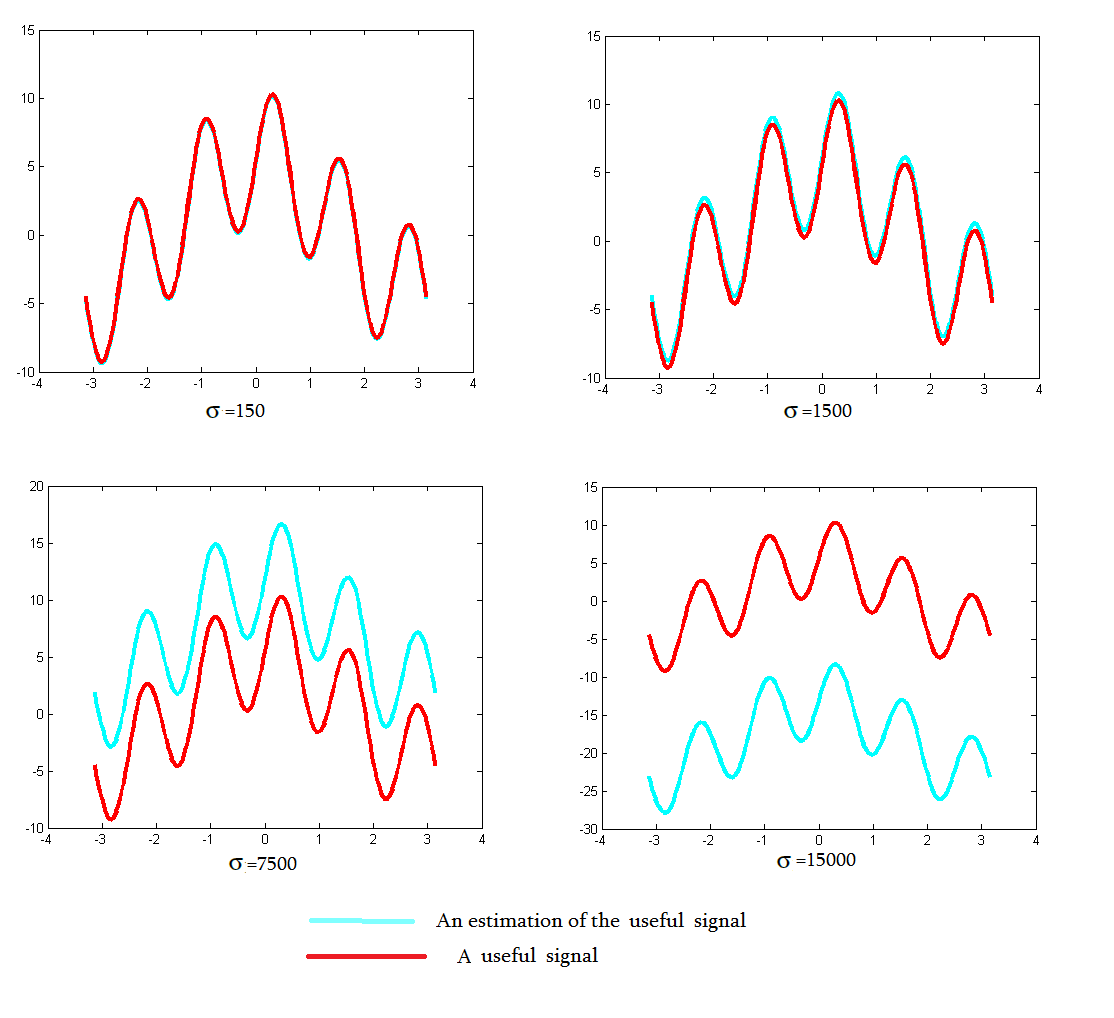}}
\caption{An estimation of the  useful signal  by using the statistic  $T_{10}$ in the Ornstein-Uhlenbeck process (4.3)-(4.4) by the sample of the size  $10$  which are results of observations to  transformed signals at moment $t_0=\pi/7$   when $\sigma \in \{150; 1500; 7500; 15000 \} $
 }
\label{ris:image}
\end{figure}

$>> N=1000;  M=10;   s=150;$

$x1 = \mbox{ random}(\mbox{'Normal'},0,1,N,  1000);$

$A1=[-1,0,0,0,0,  0,0,0,0,0,  0,0,0,0,0,  0,0,0,0,0,  0,0]; $

$C1=[5,0,0,0,0,0,0,0,0,0, 0,0,0,0,0, 0,0,0,0,0]; $

$D1=[0,0,0,0,5,0,0,0,0,0, 0,0,0,0,0, 0,0,0,0,0];$

$A10=2;   C10=1; $

$\mbox{for}~ k=1:20$

$S1(k)=A10;   S2(k)=A20;$

$\mbox{for} ~n=1:10$

$S1(k)=S1(k)+(-1)^{n}*A1(2*n)*k^{2*n};$

$\mbox{end}$

$\mbox{end}$

$\mbox{for} ~k=1:20$

$O1(k)=0;$

$\mbox{end}$

$\mbox{for} ~k=1:20$

$\mbox{for} ~n=1:10$

$O1(k)=O1(k) +(-1)^{n}*A1(2*n+1)*k^{2*n+1};$

$\mbox{end}$

$\mbox{end}$

$T1=pi/7;$

$X1=-pi:( pi/100):pi;$

$\mbox{for} ~ m=1:M$

$Z_{m}=0.5* C10*exp(T1*A10)  + 2/sqrt(2*A10)*exp(-A10*T1)*x1(m,1)*(exp(2*A10*T1)-1);$

$\mbox{end}$

$\mbox{for} ~ m=1:M$

$\mbox{for} ~ k=1:20$

$Z_{m}=Z_{m}+ C1(k)*exp(T1*S1(k)).*cos(X1.*k).* cos(T1*O1(k))+D1(k)*exp(T1*S1(k)).*cos(X1.*k).* sin(T1*O1(k))+ $

$D1(k)*exp(T1*S1(k)).* sin(X1.*k).* cos(T1*O1(k))- C1(k)*exp(T1*S1(k)).* sin(X1.*k).* sin(T1*O1(k))+ $

$ s*2/sqrt(2*A10)*exp(-A10*T1)*sqrt(2)*x1(m,k+1)*sin(pi*k*(exp(2*A10*T1)-1))/(pi*k);$

$\mbox{end}$

$\mbox{end}$

$W=0;$

$\mbox{for}~ m=1:M$

$W= W+Z_{m};$

$\mbox{end}$

$W=W/M;$

$c=0;$

$y=W;$

$\mbox{for} ~s=1:200$

$c=c+(2/(2*pi*200))*y(s);$

$\mbox{end}$

$\mbox{for} ~ m=1:20$

$a_{m}= 0; $

$b_{m}=0;$

$\mbox{end}$

$\mbox{for}  ~m=1:20$

$\mbox{for}~ k=1:200$

$a_{m}= a_{m}+(1/100)*y(k)*cos(m*X1(k));$

$b_{m}=b_{m}+(1/100)*y(k)*sin(m*X1(k));$

$\mbox{end}$

$\mbox{end}$

$Y=c*exp(-T1*A10)/2;$

$\mbox{for} ~k=1:20$

$Y=Y+exp(-S1(k)* T1) * (  (a_{k}*cos(-O1(k)* T1) + b_{k}*sin(-O1(k)* T1) ) *cos(k*X1)+( b_{k}*cos(-O1(k)* T1) -a_{k}*sin(-O1(k)* T1) )*sin(k*X1) )$

$\mbox{end}$

$Y3=5*cos(X1)+5*sin(5*X1)+C10/2;$

$\mbox{plot}(X1,Y,'c',X1,Y3,'r', \mbox{'LineWidth'},3)$

\end{ex}

\begin{figure}[h]
\center{\includegraphics[width=1.05\linewidth]{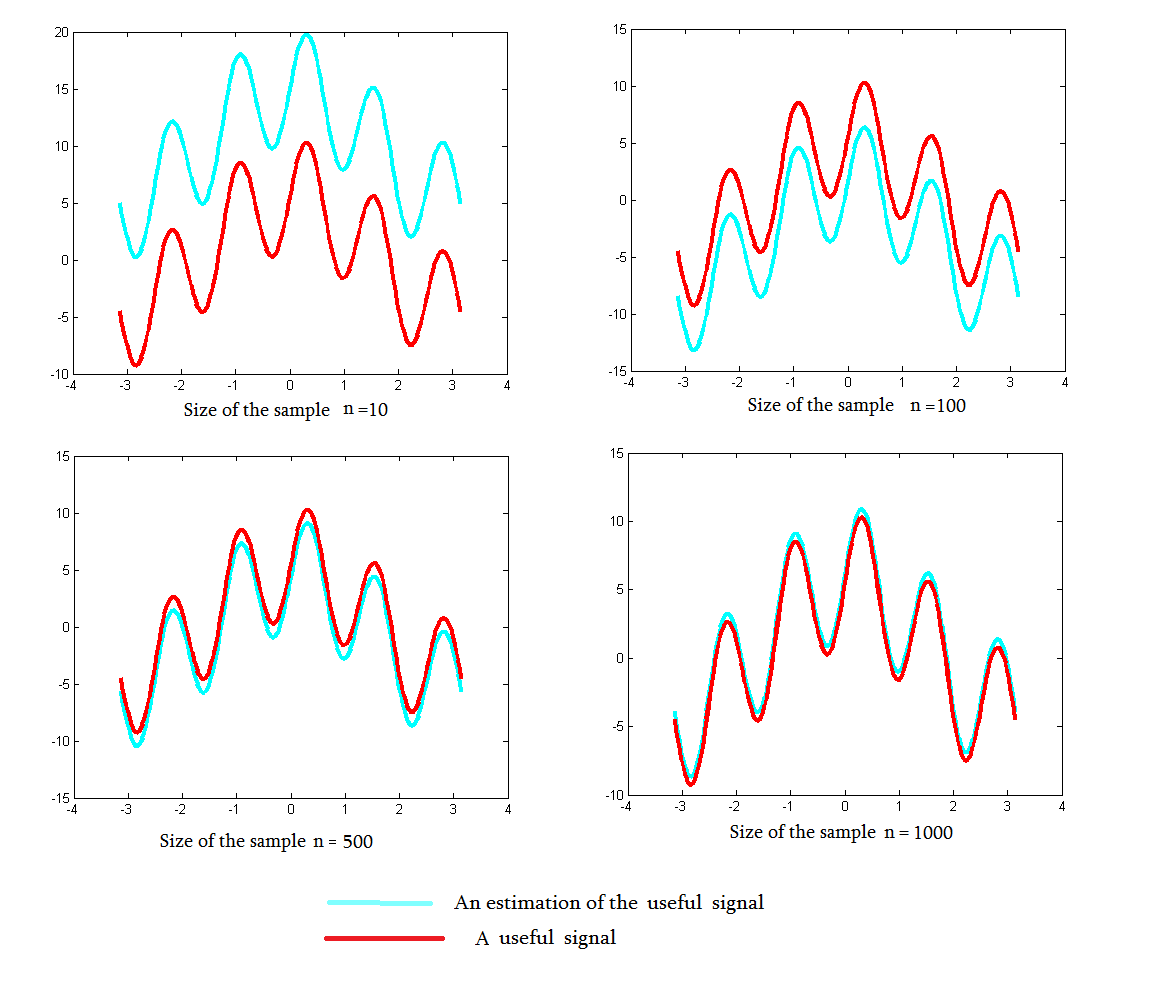}}
\caption{An estimation of the  useful signal  by  using the statistic $T_n$  in the Ornstein-Uhlenbeck process (4.5)-(4.6) by the  increase of a   size  of  the sample  when $\sigma=15000$
 }
\label{ris:image}
\end{figure}

\section{Discussion and conclusion}

If a transmitting  process  of  a  useful signal $\theta \in \Theta$ is  described by the  Ornstein-Uhlenbeck  stochastic system  (1.1)   and we  have
results of observations  $(Z_k)_{1 \le k \le n}$  on transformed signals at  any moment $t_0>0$,  then following Theorem 3.1, by  using  the statistic $T_n$
we  can  restore  $\theta$.

 Programs in Matlab prepared  in the  present  paper can be  described as follows:

(i) A programm  in Matlab  from  Example 4.1 demonstrates  animation of  a particular case  of the  Ornstein-Uhlenbeck  stochastic system  (1.1)   which are defined by  (4.1)-(4.2)(see  Figure 1).

(ii) A programm  in Matlab  from  Example 4.2  draws  and presents  a sample  $(Z_k)_{1 \le k \le n}$ of the size $n$  which consists from  results of observations to $n$  independent  transformed  signals  at moment $t=\pi/7$ when a transmitting  process of  a  useful signal $\theta$ is  described by the  Ornstein-Uhlenbeck  stochastic system  (1.1)  defined by  (4.3)-(4.4) (see  Example  4.2  and  Figure 2  for which $n=4$ ).

(iii) A programm  in Matlab  from  Example 4.3  draws  the  value  of the statistic $T_{n}$ (which is in  $C[-\pi,\pi[$)   calculated  for sample  $(Z_k)_{1 \le k \le n}$ of the size $n$  which consists from  results of observations  to $n$    independent  transformed  signals  at moment $t=\pi/7$ when a transmitting process  of  a  useful signal $\theta$ is  described by the  Ornstein-Uhlenbeck  stochastic system  (1.1)  defined by  (4.5)-(4.6). (see  Example  4.3  and  Figure 3 ).

From Figure  3  we see that the reduction of  the  parameter  $\sigma$  in   (4.3), for the  fixed size of the  sample(here, $n=10$)   increases the accuracy of the estimation of the  useful  signal   which  seems naturally.

Similarly, from Figure 4  we see   that an  increase    of  the size of the  sample, for a fixed  big  value of the parameter $\sigma$  in   (4.3)( here, $\sigma=1500$ ), also   increases the accuracy of the estimation of the  useful  signal   which  do not contradicts to the result of Theorem 3.1.

\bibliographystyle{amsplain}

\begin{thebibliography}{99}


\bibitem{Gant66} Gantmacher F. R.:   {\em Theorie des matrices.}  Tome 1: Theorie gen-
erale. (French) Traduit du Russe par Ch. Sarthou. Collection Universitaire
de Mathematiques, No. 18 Dunod, 1966.



\bibitem{GarbaczewskiOlkiewicz2000} Garbaczewski, P., Olkiewicz, R.: Ornstein-Uhlenbeck-Cauchy process,
J. Math. Phys.,  41(2000), 6843–6860.







\bibitem{Gillespie1996}
Gillespie, D. T.:  Exact numerical simulation of the Ornstein-Uhlenbeck process and its integral.
Physical review E 54, (1996).  no. 2: 2084–2091.


\bibitem{OrnsteinUhlenbeck1930}
Ornstein, L. S., Uhlenbeck, G. E.:   On the Theory of the Brownian Motion. Physical Review
36,(1930). no. 5: 823.    doi:10.1103/PhysRev.36.823.


\bibitem{Labadze2016} Labadze, L.,Pantsulaia, G.:
Estimation  of the parameters of the  Ornstein-Uhlenbeck's  process.

{https://arxiv.org/pdf/1608.04507v3.pdf}{destination}




\bibitem{Pantsulaia2016-1} Pantsulaia, G. R., Giorgadze, G.P.: On a Linear Partial Differential Equation of the Higher Order in TwoVariables with Initial Condition Whose
Coefficients are Real-valued Simple Step Functions,  J. Partial Diff. Eqs.,  29 (2016) . No. 1, 1-13


\bibitem{Pantsulaia2016}Labadze, L., Saatashvili, G.,  Pantsulaia, G.: Infinite-sample consistent estimations of parameters of the Wiener process with drift.

{https://arxiv.org/pdf/1611.01119v2.pdf}{destination}


\bibitem{Protter2004}
Protter, P.: Stochastic integration and differential equations, Springer-Verlag, Berlin,
2004.




\bibitem{Sh80}
Shiryaev, A.N.:  {\em Probability} (in Russian), Izd.``Nauka", Moscow, 1980.

\bibitem{Smith2010}
Smith, William.:   On the Simulation and Estimation of the Mean-Reverting Ornstein-Uhlenbeck Process, Especially as Applied to Commodities Markets and Modelling,
Verson 1.01 (February), 2010.

{https://commoditymodels.files.wordpress.com/2010/02/estimating-the-parameters-of-a-mean-reverting-ornstein-uhlenbeck-process1.pdf}{destination}




\bibitem{Ibram80}
Ibramkhallilov, I.Sh., Skorokhod, A.V.: {\em On well--off estimates of parameters
of stochastic processes} (in Russian),  Kiev, 1980.



\end{thebibliography}

\end{document}